\documentclass[11pt]{article}

\usepackage[T1]{fontenc}
\usepackage{lmodern}
\usepackage{amsmath,amssymb,amsthm,mathtools}
\usepackage[a4paper,margin=1in]{geometry}
\usepackage{microtype}
\usepackage{booktabs,tabularx,array}
\usepackage[colorlinks=true,linkcolor=blue,citecolor=blue,urlcolor=blue]{hyperref}

\newtheorem{theorem}{Theorem}
\newtheorem{corollary}{Corollary}
\newtheorem{proposition}{Proposition}
\newtheorem{lemma}{Lemma}
\theoremstyle{definition}
\newtheorem{definition}{Definition}
\theoremstyle{remark}
\newtheorem{remark}{Remark}

\newcommand{\R}{\mathbb R}
\newcommand{\Z}{\mathbb Z}
\newcommand{\Q}{\mathbb Q}

\newcommand{\Orb}{\mathcal O}
\newcommand{\sk}{\sigma_d^{>}}
\newcommand{\card}{\#}

\DeclareMathOperator{\dist}{dist}

\newcolumntype{Y}{>{\raggedright\arraybackslash}X}
\newcolumntype{L}[1]{>{\raggedright\arraybackslash}p{#1}}
\renewcommand{\arraystretch}{1.12}

\title{Kronecker sequences on mixed homogeneous spaces:\\
nearest-neighbour distances and best returns}
\author{Evgeniy Zorin\footnote{Department of Mathematics, University of York,
Heslington, York, YO10 5DD, England.\\
Email: \texttt{evgeniy.zorin@york.ac.uk}}}
\date{}
\hypersetup{
  pdftitle={Kronecker sequences on mixed homogeneous spaces: nearest-neighbour distances and best returns},
  pdfauthor={Evgeniy Zorin}
}

\begin{document}
\maketitle

\begin{abstract}
We study nearest-neighbour distances and strict best returns of cyclic
orbits on $(V\times G)/\Lambda$, where $V$ is a finite-dimensional real
normed space, $G$ has a bi-invariant ultrametric, and $\Lambda$ is an
arbitrary closed subgroup. For the quotient of the maximum product metric,
the number of distinct nearest-neighbour distances in every finite orbit
segment is at most $P(\overline V)+1$, where $P$ is an open-ball packing
constant and $\overline V$ is obtained by factoring out the continuous real part
of $\Lambda$. For inner-product norms the bound depends only on the span
of the projected subgroup. This gives $2^d+1$ for the maximum norm on
$\mathbb R^d$ and strong-kissing-number bounds for inner-product norms,
extending the corresponding real and adelic distance bounds.

Our second main result is an exact abelianization principle: every such
cyclic orbit is isometric, with its integer indexing preserved, to a cyclic
orbit on an abelian mixed quotient $(V\times U)/\Lambda_{\mathrm{ab}}$.
The construction preserves the reduced and effective real factors.
Thus noncommutativity and non-normality enlarge the class of ambient
presentations but not the class of cyclic-orbit metrics.

The record-count identity and its growth consequences require only an
isometric embedding of a metric space. For mixed homogeneous spaces we
prove packing recurrences and, using abelianization, the sharp three-step
and five-step recurrences in effective Euclidean dimensions one and two,
without a normalizer assumption. We also extend the purely real recurrence
to noncompact quotients and give counterexamples showing why the stronger
real recurrence can fail after an ultrametric factor is added.
\end{abstract}

\section{Introduction}\label{sec:introduction}

\subsection{From the three-gap theorem to homogeneous spaces}

The classical three-gap theorem asserts that a finite Kronecker sequence on
the circle
\[
 \mathbb T^1=\mathbb R/\mathbb Z,\qquad 0,\alpha,\ldots,(N-1)\alpha,
\]
cuts the circle into arcs of at most three different lengths; see
\cite{MarklofStrombergsson2017,Sos1958}. In particular, the distances from
these points to their nearest neighbours take at most three values. This
metric consequence does not use the cyclic order of the circle and therefore
suggests extensions well beyond one-dimensional tori.

On real tori, Chevallier \cite{Chevallier1996,ChevallierApproximations1996}
related nearest-neighbour distances to best simultaneous approximations and
proved a five-distance bound for the maximum metric on $\mathbb R^2/\mathbb
Z^2$. Haynes and Ramirez \cite{HaynesRamirez2021} established the bound
$2^d+1$ for the maximum metric on $\mathbb R^d/\mathcal L$, where $\mathcal
L$ is an arbitrary full-rank lattice, and proved optimality in dimensions two
and three. For Euclidean metrics, Haynes and Marklof
\cite{HaynesMarklof2022} proved the sharp planar five-distance theorem and
obtained higher-dimensional bounds in terms of kissing numbers. Romanov's
denominator estimates \cite{Romanov2006} give another route to such bounds;
the relation with Chevallier's work is developed explicitly by Shutov
\cite{Shutov2024}.

Shulga \cite{Shulga2026} recently proved $g_N\leq2^d+1$ for every
inner-product metric on a real flat $d$-torus. In dimension three this settles
the nine-distance conjecture of Haynes and Marklof; Dettmann's example
\cite{Dettmann2025} shows that nine is optimal. Shulga also formulates the
earlier angular argument with the strong kissing number, retaining the strict
angular separation in the proof. Although his denominator-growth theorem is
stated for full-rank lattices, we show below that its recurrence consequence
extends to arbitrary, possibly noncompact, purely real quotients; the relevant
index is then governed by the rank of the discrete part of the subgroup.

A different extension is furnished by Das and Haynes
\cite{DasHaynes2023}, who proved a sharp three-distance theorem for adelic
tori associated with any non-empty set of primes, finite or infinite. Their
spaces combine one real coordinate with an ultrametric component. All of the
spaces just mentioned are quotient groups, and all groups occurring in these
previous results are abelian.

We study these questions on mixed homogeneous spaces. Let $V$ be a
finite-dimensional real normed space, let $G$ be a group endowed with a
bi-invariant ultrametric, and put $H=V\times G$. For an arbitrary closed
subgroup $\Lambda\leq H$, the right-coset space $X=H/\Lambda$ carries the
quotient of the maximum product metric. Normality is not assumed, so the
group operation on $H$ need not descend to $X$. Nevertheless, left
multiplication acts isometrically, and we may study the cyclic orbit
\[
    \Lambda,\ a\Lambda,\ a^2\Lambda,\ldots\qquad(a\in H).
\]

Our first main result bounds the number of nearest-neighbour distances by
a packing constant of the real factor alone. Continuous real directions
contained in $\Lambda$ can be removed. For inner-product norms, directions
outside the span of the projected subgroup can also be omitted from the
packing count, giving an effective dimension which may be smaller than
$\dim V$. The ultrametric factor contributes no dimension to this bound.

Our second main result explains the precise relation with the abelian
setting. For each fixed orbit generator we construct an abelian ultrametric
group $U$ and a closed subgroup $\Lambda_{\mathrm{ab}}\leq V\times U$ such
that the entire indexed cyclic orbit has exactly the same metric as an
orbit on $(V\times U)/\Lambda_{\mathrm{ab}}$. The real factor, its norm,
and the reduced and effective real data are preserved. The model can
depend on the generator: this is a theorem about individual orbits, not
an identification of the whole homogeneous space with an abelian quotient.
It shows that the nonabelian presentations give natural new settings for
the problem without introducing new cyclic-orbit metrics in this class.
The Weyl-twisted suspensions below illustrate both sides of this distinction.

The third theme is the relation between distance counts and denominator
growth. The record-count identity is valid for an isometric embedding;
surjectivity is unnecessary. Direct separation gives packing recurrences
on every mixed homogeneous space. Abelianization also transfers the sharper
planar recurrence, whose proof combines three small lifts additively.
Thus the sharp three-step and five-step recurrences hold in effective
Euclidean dimensions one and two for every closed $\Lambda$, without
normality or a normalizer hypothesis. The remaining distinction is between
purely real and mixed geometry: stronger real recurrences can fail even
in abelian mixed quotients. We retain a direct homogeneous-space proof of
the distance theorem because it makes the geometric role of the
ultrametric factor especially transparent.

\subsection{The homogeneous metric setting and uniform distance bounds}
\label{subsec:bounds}

Let $V$ be a finite-dimensional real normed space of dimension $d\geq0$. Let $G$ be a group with
identity $e_G$ and a bi-invariant ultrametric $d_G$, so that
\[
 d_G(kg,kh)=d_G(g,h)=d_G(gk,hk)
 \qquad(g,h,k\in G).
\]
Write $|g|_G=d_G(e_G,g)$. The ultrametric inequality gives
\begin{equation}\label{eq:ultrametric}
 |gh|_G\leq\max\{|g|_G,|h|_G\},
 \qquad |g^{-1}|_G=|g|_G.
\end{equation}
On
\[
    H=V\times G
\]
write $e_H=(0,e_G)$ and put
\begin{equation}\label{eq:ambient}
 D\bigl((v,g),(w,h)\bigr)
 =\max\{\|v-w\|_V,d_G(g,h)\}.
\end{equation}
This is a bi-invariant metric on $H$.

Let $\Lambda\leq H$ be a closed subgroup, not necessarily normal, and let
\begin{equation}\label{eq:quotient}
 X=H/\Lambda,
 \qquad
 d_X(x\Lambda,y\Lambda)
   =\inf_{\lambda\in\Lambda}D(x,y\lambda).
\end{equation}
Bi-invariance makes the formula independent of the chosen representatives,
and closedness of $\Lambda$ makes $d_X$ a metric. Left multiplication by $H$
acts isometrically on $X$. Neither compactness of $X$ nor discreteness or
normality of $\Lambda$ is assumed. The infimum in \eqref{eq:quotient} need not
be attained.

Fix $a\in H$. For $N\geq2$, consider the orbit set
\[
 \Orb_N(a,\Lambda)=\{\Lambda,a\Lambda,\ldots,a^{N-1}\Lambda\},
\]
with repetitions discarded. If this set has at least two points, put
\begin{equation}\label{eq:distances}
 \delta_N(x)=\min_{\substack{y\in\Orb_N(a,\Lambda)\\y\ne x}}d_X(x,y),
 \qquad
 g_N(a,\Lambda)=\card\{\delta_N(x):x\in\Orb_N(a,\Lambda)\}.
\end{equation}
We set $g_N=0$ for a singleton.

\begin{definition}[Packing number]\label{def:packing}
For a finite-dimensional real normed space $V$, let
\begin{equation}\label{eq:packing}
 P(V)=\max\bigl\{\card E:E\subset B_V(0,1),\
       \|v-w\|_V\geq1\ \text{for }v\ne w\bigr\},
\end{equation}
where $B_V(0,1)=\{v:\|v\|_V<1\}$ is the \emph{open} unit ball.
This maximum is finite; in fact $P(V)<3^d$ when $d\geq1$, while
$P(\{0\})=1$.
\end{definition}

\begin{remark}
If $W\subset V$ is a linear subspace and $V/W$ carries the quotient norm,
then
\[
    P(V/W)\leq P(V).
\]
Indeed, any admissible configuration in the open unit ball of $V/W$ can be
lifted to the open unit ball of $V$, and quotient distances do not exceed the
distances between the chosen lifts.
\end{remark}

\begin{definition}[Reduced real factor]\label{def:reduced-space}
Let
\[
    W=\{v\in V:(tv,e_G)\in\Lambda\ \text{for every }t\in\mathbb R\}.
\]
Then $W$ is a real vector subspace. Put
\[
    \overline V=V/W
\]
and equip $\overline V$ with the quotient norm. Let $\overline\Lambda$ be the
image of $\Lambda$ in $\overline V\times G$. We call $\overline V$ the
\emph{reduced real factor} of the homogeneous space.
\end{definition}

\begin{definition}[Effective real factor]\label{def:effective-space}
Suppose that $V$ is an inner-product space. Identify $\overline V$ with
$W^\perp$, and let $\pi:V\to W^\perp$ be orthogonal projection. Put
\[
    E_\Lambda
    =\operatorname{span}_{\mathbb R}
      \{\pi(v):(v,g)\in\Lambda\},
    \qquad
    m_\Lambda=\dim E_\Lambda.
\]
We call $E_\Lambda$ the \emph{effective real factor} and $m_\Lambda$ the
\emph{effective real dimension}. The space $E_\Lambda$ carries the induced
inner product.
\end{definition}

Our first main result is the homogeneous-space form of the distance theorem.
\begin{theorem}[Effective packing bound on homogeneous spaces]
\label{thm:packing}
For every closed subgroup $\Lambda\leq V\times G$, every $a\in V\times G$
and every $N\geq2$,
\[
    g_N(a,\Lambda)\leq P(\overline V)+1.
\]
If $V$ is an inner-product space, then
\[
    g_N(a,\Lambda)\leq P(E_\Lambda)+1.
\]
\end{theorem}

The abelian specialization provides a convenient formulation containing
the real and adelic examples.
\begin{corollary}[Abelian mixed quotients]\label{cor:mixed-groups}
Let $U$ be an abelian group with a translation-invariant ultrametric, let
$\Lambda\leq V\times U$ be closed, and put $X=(V\times U)/\Lambda$.
For every $\alpha\in X$ and $N\geq2$, the conclusions of
Theorem~\ref{thm:packing} hold for the orbit
$0,\alpha,\ldots,(N-1)\alpha$, with repetitions discarded.
\end{corollary}

\begin{remark}[Nonabelian quotient groups]\label{rem:nonabelian}
Theorem~\ref{thm:packing} also includes quotient groups with a nonabelian
ultrametric factor whenever $\Lambda$ is normal. Its proof requires neither
commutativity of $G$ nor normality of $\Lambda$. The next result gives a
converse at the level of individual cyclic-orbit metrics: every orbit
covered by the theorem has an exact realization in the abelian subclass.
\end{remark}

\subsection{Abelian models for cyclic orbits}
\label{subsec:abelianization}

The following is a structural statement in the same mixed class, rather
than merely the observation that a cyclic orbit is indexed by integers.
In particular, the construction does not enlarge its real factor.

\begin{theorem}[Abelianization principle]\label{thm:abelianization}
Let $\Lambda\leq V\times G$ be closed and let $a=(v,h)\in V\times G$.
There exist an abelian group $U$ with a translation-invariant ultrametric,
a closed subgroup $\Lambda_{\mathrm{ab}}\leq V\times U$, and $u\in U$
with the following properties. Equip
\[
    X_{\mathrm{ab}}=(V\times U)/\Lambda_{\mathrm{ab}}
\]
with the quotient maximum product metric and put
$\alpha=(v,u)+\Lambda_{\mathrm{ab}}$. Then
\begin{equation}\label{eq:abelian-orbit-isometry}
 d_{X_{\mathrm{ab}}}(i\alpha,j\alpha)
   =d_X(a^i\Lambda,a^j\Lambda)\qquad(i,j\in\mathbb Z).
\end{equation}
Moreover, $W_{\mathrm{ab}}=W$, so the reduced real factors agree with
their quotient norms. If $V$ is an inner-product space, then also
\[
    E_{\Lambda_{\mathrm{ab}}}=E_\Lambda,
    \qquad m_{\Lambda_{\mathrm{ab}}}=m_\Lambda.
\]
Consequently the indexed orbits are isometric, including all coincidences,
and their nearest-neighbour counts, return functions, and strict
best-return denominators agree.
\end{theorem}

Here and below an isometry between indexed orbits means the map
$a^j\Lambda\mapsto j\alpha$ on their sets of points; equality of zero
distances makes this map well defined even when an orbit is periodic.
The theorem identifies the classes of cyclic-orbit metrics arising from
mixed homogeneous spaces and abelian mixed quotients, with the reduced
real data fixed. No isomorphism of the ambient spaces or simultaneous
model for different generators is asserted. In particular, every universal
bound involving only these orbit metrics and the preserved real data holds
in the homogeneous class if and only if it holds in the abelian subclass.
The proof is given in Section~\ref{subsec:abelianization-proof}.

\paragraph{A genuinely homogeneous example: a Weyl-twisted $p$-adic suspension.}
The next family has a nonnormal, nonabelian stabilizer, so its given
homogeneous presentation is not a quotient-group presentation.
Theorem~\ref{thm:abelianization} nevertheless supplies an abelian model for
each of its cyclic orbits.

\medskip
\noindent
\textbf{(a) The $\mathrm{SL}_2$ case.}
Fix a prime $p$ and let
\[
    G=\mathrm{SL}_2(\mathbb Z_p),
\]
equipped with the bi-invariant ultrametric
\[
    d_p(g,h)=p^{-m(g,h)},\qquad
    m(g,h)=\max\{r\geq0:g\equiv h\pmod{p^r}\},
\]
with the usual convention that $d_p(g,g)=0$.
Let
\[
    T=
    \left\{
    \begin{pmatrix}
        u&0\\
        0&u^{-1}
    \end{pmatrix}
    :u\in\mathbb Z_p^\times
    \right\}
\]
be the diagonal torus, and let
\[
    w=
    \begin{pmatrix}
        0&-1\\
        1&0
    \end{pmatrix}.
\]
Then $w$ normalizes $T$ and acts on it by inversion:
\[
    wtw^{-1}=t^{-1}\qquad(t\in T).
\]
Inside
\[
    H=\mathbb R\times G
\]
consider the closed subgroup
\[
    \Lambda
    =
    \{(n,w^n t):n\in\mathbb Z,\ t\in T\}.
\]
The subgroup $\Lambda$ is closed because its real projection is
discrete and each fibre $w^nT$ is closed. It is nonabelian: indeed,
\[
    (1,w)(0,t)(1,w)^{-1}=(0,t^{-1}),
\]
and $T$ contains elements with $t^2\ne e_G$ for every prime $p$.
It is also non-normal in $H$, since
\[
    \Lambda\cap(\{0\}\times G)=\{0\}\times T
\]
and $T$ is not normal in $G$. Thus
\[
    X=(\mathbb R\times\mathrm{SL}_2(\mathbb Z_p))/\Lambda
\]
does not inherit a quotient group structure from $H$. Its natural
presentation therefore lies outside the abelian quotient setting.

Geometrically, $X$ may be viewed as the mapping torus of the Weyl
involution on the $p$-adic homogeneous space $G/T$: in a fundamental
domain for the real coordinate one has the identification
\[
    (s,gT)\sim(s+1,gwT).
\]
For an arbitrary orbit generator
\[
    a=(\alpha,A)\in\mathbb R\times\mathrm{SL}_2(\mathbb Z_p),
\]
the orbit $a^j\Lambda$ can therefore be represented, with real coordinate
in $[0,1)$, as
\[
    a^j\Lambda
    =
    \left(
        \{j\alpha\},
        A^j w^{-\lfloor j\alpha\rfloor}T
    \right),
\]
where the second coordinate is understood in $G/T$.
When $A$ does not commute with $w$, these representatives explicitly
involve noncommuting factors. This records the twisting of the homogeneous
presentation, even though its individual orbit metrics admit abelian models.

In this example the reduced real factor is $\overline V=\mathbb R$:
indeed, $\Lambda$ contains no nontrivial connected real direction.
Moreover the real projection of $\Lambda$ is $\mathbb Z$, so
\[
    E_\Lambda=\mathbb R.
\]
Since $P(\mathbb R)=2$, Theorem~\ref{thm:packing} gives
\[
    g_N(a,\Lambda)\leq3
\]
for every $a\in H$ and every $N\geq2$.  Thus an arbitrarily complicated
orbit in the $p$-adic homogeneous fibre contributes no additional
nearest-neighbour distances beyond the three allowed by the single
effective real direction.

\medskip
\noindent
\textbf{(b) Higher-rank variants.}
The same construction works in every dimension.  Let
\[
    G=\mathrm{SL}_n(\mathbb Z_p),\qquad n\geq2,
\]
let $T<G$ be the diagonal torus, and choose a representative
$w\in N_G(T)$ of a nonidentity Weyl-group element.  Then
\[
    \Lambda_w
    =
    \{(m,w^m t):m\in\mathbb Z,\ t\in T\}
    \leq \mathbb R\times G
\]
is a closed subgroup.  Whenever the action of $w$ on $T$ is nontrivial,
$\Lambda_w$ is nonabelian; and since $T$ is not normal in $G$,
$\Lambda_w$ is not normal in $\mathbb R\times G$.  The quotient
\[
    X_{n,p,w}
    =
    (\mathbb R\times\mathrm{SL}_n(\mathbb Z_p))/\Lambda_w
\]
again has a nonnormal stabilizer. Nevertheless its
effective real dimension is still one, and Theorem~\ref{thm:packing}
gives the dimension-free bound
\[
    g_N(a,\Lambda_w)\leq3
\]
for every $a\in\mathbb R\times\mathrm{SL}_n(\mathbb Z_p)$ and every
$N\geq2$.

In particular, the non-Archimedean homogeneous fibre may have arbitrarily
large dimension and noncommutative Weyl twisting, while the uniform
nearest-neighbour bound remains governed solely by the single real
direction.

\begin{remark}[An explicit abelian model for the suspension]
\label{rem:weyl-abelian}
For the $\mathrm{SL}_2$ example and a fixed $a=(\alpha,A)$, the construction
in Theorem~\ref{thm:abelianization} takes a particularly concrete form.
On $\mathbb Z^2$ put
\[
 \Phi(m,q)=\inf_{t\in T}d_p(A^q,w^m t),\qquad
 N_0=\{(m,q):\Phi(m,q)=0\},\qquad U_A=\mathbb Z^2/N_0.
\]
Lemma~\ref{lem:Phi-ultrametric} shows that
$|[(m,q)]|_{U_A}=\Phi(m,q)$ defines an invariant ultrametric.
Let
\[
 \Lambda_A=\{(m,[(-m,0)]):m\in\mathbb Z\},\qquad
 \alpha_A=(\alpha,[(0,1)])+\Lambda_A.
\]
The graph $\Lambda_A$ is closed in $\mathbb R\times U_A$, since its real
projection is discrete. The return function of $\alpha_A$ is
\[
 \inf_{m\in\mathbb Z}\max\{|q\alpha-m|,\Phi(m,q)\},
\]
which is exactly the return function of $a$ on $X$. Hence every finite
orbit segment, every nearest-neighbour distance, and every strict best
return is preserved. Since $w^2\in T$, the class $[(1,0)]$ has order at
most two; $U_A$ is algebraically generated by this class and $[(0,1)]$.
For the higher-rank examples the same construction applies with the
corresponding $T$ and $w$.
Thus the example is genuinely nonabelian as a homogeneous presentation,
while its cyclic-orbit data have an exact abelian realization.
\end{remark}

\subsection{Quantitative specializations and earlier distance bounds}
\label{subsec:specializations}

\begin{definition}[Strong kissing number]\label{def:kissing}
For $d\geq1$, the strong kissing number, also called the
\emph{strict kissing number}, is
\begin{equation}\label{eq:kissing}
 \sigma_d^{>}=\max\bigl\{\card C:C\subset S^{d-1},\
           \langle\xi,\eta\rangle<\tfrac12\ \text{for }\xi\ne\eta\bigr\}.
\end{equation}
The ordinary kissing number $K_d$ permits equality in the inner-product
condition, so $\sigma_d^{>}\leq K_d$. For the Euclidean ball,
\eqref{eq:kissing} is also called the strict Hadwiger number
\cite{NaszodiPachSwanepoel2017}.
\end{definition}

\begin{corollary}[Inner-product norm]\label{cor:euclidean}
Suppose that the norm of $V$ is induced by an inner product. If
$m_\Lambda\geq1$, then
\[
    g_N(a,\Lambda)\leq \sigma_{m_\Lambda}^{>}+1.
\]
If $m_\Lambda=0$, then
\[
    g_N(a,\Lambda)\leq2.
\]
In particular, for effective dimensions one, two and three the respective
bounds are $3$, $6$ and $13$, independently of the ambient dimension.
\end{corollary}

\begin{corollary}[Maximum norm]\label{cor:sup}
Suppose that the unit ball of $V$ is a linear image of a $d$-dimensional
cube. Then
\[
    g_N(a,\Lambda)\leq2^d+1.
\]
If, in addition, the unit ball of $\overline V$ is a linear image of a cube,
choose corresponding cube coordinates and let $k$ be the number of
cube-coordinate directions occurring in the real projection of
$\overline\Lambda$. Then
\[
    g_N(a,\Lambda)\leq2^k+1,
    \qquad k\leq\dim\overline V\leq d.
\]
\end{corollary}

\begin{remark}[Effective factors]\label{rem:effective-factor}
The second conclusion of Theorem~\ref{thm:packing} is a special case of a
slightly more general statement. Suppose that
\[
    \overline V=E\oplus F,
    \qquad
    \overline\Lambda\subset E\times G,
\]
and that either the decomposition is orthogonal for an inner-product norm or
\[
    \|e+f\|_{\overline V}=\max\{\|e\|_E,\|f\|_F\}.
\]
Then
\[
    g_N(a,\Lambda)\leq P(E)+1.
\]
For an inner-product norm one may take $E=E_\Lambda$ and
$F=E_\Lambda^\perp$.
\end{remark}

The constants in the effective refinements reflect only the geometry of the
relevant real factor. A $k$-dimensional cube has $2^k$ sign regions, in each
of which any two points have distance less than one. For an inner-product
norm, radial projection of a $1$-separated set in the effective real ball
produces a strong kissing configuration. The ultrametric coordinate disappears
from the separation count because two points lying in the same ultrametric
ball of radius $r$ are at distance $<r$ from one another.

\begin{remark}[The purely ultrametric case]\label{rem:purely-ultrametric}
Theorem~\ref{thm:packing} also holds for $V=\{0\}$. It gives
$g_N\leq2$ for every cyclic orbit on $G/\Lambda$, where $G$ is any group with
a bi-invariant ultrametric and $\Lambda\leq G$ is closed. The bound is sharp:
for $G=\mathbb Z_2$, $\Lambda=\{0\}$, $a=1$ and $N=3$, the points $0,1,2$
have nearest-neighbour distances $1/2,1,1/2$.
\end{remark}

We now explain how the earlier uniform upper bounds are contained in the
abelian mixed-quotient specialization. For Chevallier and Haynes--Ramirez take
$U=\{e\}$, $V=(\mathbb R^d,\|\cdot\|_\infty)$ and a lattice $\Lambda$.
Corollary~\ref{cor:mixed-groups}, together with $P(\mathbb
R^d,\|\cdot\|_\infty)=2^d$, gives $2^d+1$. For the adelic theorem of
Das--Haynes take
\[
 V=\mathbb R,\qquad
 U=\prod_{p\in\mathcal P}'\mathbb Q_p,\qquad
 \Lambda=\operatorname{diag}\mathbb Z[1/p:p\in\mathcal P].
\]
For finite $\mathcal P$ use $|u|_U=\max_{p\in\mathcal P}|u_p|_p$; for
infinite $\mathcal P$ use
$|u|_U=\sup_{p\in\mathcal P}p^{-1}|u_p|_p$. These are precisely the
ultrametric components of the metrics in \cite{DasHaynes2023}, including the
weights for infinite $\mathcal P$. For completeness, the diagonal subgroup is
closed: if a nonzero $\gamma=m/D$ is written in lowest terms and $D>1$, a
prime $p\mid D$ gives $p^{-1}|\gamma|_p\geq1$; if $D=1$, then
$|\gamma|\geq1$. Thus distinct diagonal elements have distance at least one.
Corollary~\ref{cor:mixed-groups} with one real dimension gives $g_N\leq3$.
Likewise the strong-kissing-number bound for real inner-product tori is the
case $U=\{e\}$ of the same corollary, together with the Euclidean value of
$P(V)$. In this sense the previous uniform upper-bound theorems arise already
from the abelian mixed-quotient corollary of the homogeneous-space result.

The sharp examples in the real maximum-metric setting show that the bound
$2^k+1$ is optimal for $k=2,3$.

\begin{remark}[Conventions for distances on a torus]\label{rem:convention}
Some real-torus papers, including \cite{HaynesRamirez2021,HaynesMarklof2022,
Shulga2026}, minimize over all nonzero lifted displacements and allow a
nonzero lattice displacement from a point to itself. If
$s=\min_{0\ne\ell\in\mathcal L}\|\ell\|$, this replaces each
$\delta_N(x)$ by $\min\{s,\delta_N(x)\}$. Such a common truncation cannot
increase the number of distinct values. A singleton contributes one value
under that convention. Thus all our upper bounds also apply with this
convention.
\end{remark}

\begin{remark}
The ambient and reduced packing bounds are closely related to the general
metric principle of Biringer and Schmidt \cite[Lemma~1]{BiringerSchmidt2008}
for finite orbit segments of isometries. The effective-factor refinement uses
in addition the direct-product structure and the affine dependence of the real
coordinate of $a^n\Lambda$ on $n$.
\end{remark}

The strong kissing numbers are difficult to determine: already in dimension
four, it is not known whether $\sigma_4^{>}<24$. Even in effective dimensions
two and three, where their exact values are known,
Corollary~\ref{cor:euclidean} gives $6$ and $13$, respectively. For purely
real quotients of ranks two and three,
Proposition~\ref{prop:shulga-real} below gives the stronger bounds $5$ and
$9$.

This comparison reverses in sufficiently high effective dimensions. The
uniqueness of the $E_8$ and Leech kissing configurations
\cite{BannaiSloane1981,ConwaySloane1999}, each containing pairs at angle
$60^\circ$, implies
\[
 \sigma_8^{>}\leq239,\qquad \sigma_{24}^{>}\leq196559.
\]
Consequently, Corollary~\ref{cor:euclidean} gives
\begin{equation}\label{eq:exceptional-dimensions}
 g_N\leq240\quad(m_\Lambda=8),\qquad
 g_N\leq196560\quad(m_\Lambda=24).
\end{equation}
More generally, the Kabatiansky--Levenshtein estimate
\cite{KabatianskyLevenshtein1978} gives
\begin{equation}\label{eq:asymptotic}
 g_N\leq1+2^{(c_{\mathrm{KL}}+o(1))m_\Lambda},
 \qquad c_{\mathrm{KL}}=0.40141\ldots,
 \quad m_\Lambda\longrightarrow\infty.
\end{equation}

\subsection{Best returns and sharp low-dimensional recurrences}
\label{subsec:returns}

The record-count mechanism needs only an isometric embedding. Let $(Y,d)$
be a metric space and let $T:Y\to Y$ satisfy
\[
    d(Tx,Ty)=d(x,y)\qquad(x,y\in Y).
\]
Thus $T$ is injective, but need not be surjective. Fix $x_0\in Y$ and use
only the forward iterates $T^q x_0$, $q\geq0$. Put
\begin{equation}\label{eq:returns}
 \rho(q)=d(T^q x_0,x_0),\qquad
 R(Q)=\min_{1\leq q\leq Q}\rho(q)\quad(q,Q\in\mathbb N).
\end{equation}
For the finite orbit set $\{x_0,Tx_0,\ldots,T^{N-1}x_0\}$, with repetitions discarded, define nearest-neighbour distances and their number $g_N$ exactly as in \eqref{eq:distances}. Its \emph{strict best-return denominators} are
\begin{equation}\label{eq:records}
 q_1=1,\qquad q_{n+1}=\min\{q>q_n:\rho(q)<\rho(q_n)\},
 \qquad r_n=\rho(q_n).
\end{equation}
The sequence is allowed to terminate, and a zero return is retained as its
final term. For our homogeneous spaces we take $Y=X$, $x_0=\Lambda$ and
$T(x\Lambda)=ax\Lambda$. In that setting $T$ is bijective and the notation
$\rho(q)$ also makes sense for $q\in\mathbb Z$, with
$\rho(-q)=\rho(q)$ and $\rho(0)=0$.

If the first $N$ orbit points are distinct, then
\begin{equation}\label{eq:intro-count}
 g_N=1+\card\{n:\lfloor N/2\rfloor<q_n\leq N-1\}.
\end{equation}
Only the stationarity of the orbit distances is used:
\[
 d(T^i x_0,T^j x_0)=d(x_0,T^{|j-i|}x_0)\qquad(i,j\geq0).
\]
Thus it would also suffice to require distance preservation on the
forward orbit alone. The next theorem records the growth consequence.

\begin{theorem}[From return-time growth to distances]\label{thm:growth}
Let $T$ be an isometric embedding of a metric space into itself and let
$x_0$ be a point.
Suppose that an integer $M\geq1$ satisfies
\begin{equation}\label{eq:doubling}
 q_{n+M}\geq2q_n
\end{equation}
whenever $q_{n+M}$ exists. Then every finite orbit segment has at most $M+1$
distinct nearest-neighbour distances. In particular, the conclusion holds
under the stronger hypothesis
\begin{equation}\label{eq:recurrence-M}
 q_{n+M}\geq q_n+q_{n+1}.
\end{equation}
\end{theorem}

In the real lattice setting, denominator growth has a substantial history.
Lagarias \cite{Lagarias1982}, Theorems~2.2--2.3, proved the $2^d$-step
recurrence for the supremum norm under the hypotheses stated there, as well
as a $2^{d+1}$-step inequality for arbitrary norms. Romanov
\cite{Romanov2006} obtained the four-step Euclidean recurrence in dimension
two. For every $d$-dimensional inner-product space and every full-rank
lattice, Shulga \cite{Shulga2026}, Theorem~1.4, proves
\begin{equation}\label{eq:shulga}
 q_{n+2^d}\geq
 \min\{2q_{n+1},\,q_n+q_{n+2^{d-1}}\}
 \geq q_n+q_{n+1}.
\end{equation}
Theorem~\ref{thm:growth} then gives $g_N\leq2^d+1$.

Although Shulga states his theorem for full-rank lattices, compactness is not
essential to the recurrence consequence.
\begin{proposition}[Projection extension of Shulga's recurrence]
\label{prop:shulga-real}
Suppose that $G=\{e_G\}$, that $V$ is a finite-dimensional inner-product
space and that $\Lambda\subset V$ is closed. Let $\Lambda^\circ$ be the
identity component of $\Lambda$ and put
\[
 r=\operatorname{rank}(\Lambda/\Lambda^\circ).
\]
If $r\geq1$, then
\begin{equation}\label{eq:shulga-real}
 q_{n+2^r}\geq q_n+q_{n+1}
\end{equation}
whenever the left-hand side is defined. Consequently,
$g_N\leq2^r+1$ for every $N\geq2$. If $r=0$, there is no strict
best-return denominator after $q_1=1$, and $g_N\leq1$.
\end{proposition}

In the purely real inner-product setting one has
\[
    m_\Lambda=\operatorname{rank}(\Lambda/\Lambda^\circ)=r.
\]
Thus Corollary~\ref{cor:euclidean} gives the strong-kissing bound
$\sigma_r^{>}+1$, while Proposition~\ref{prop:shulga-real} gives
$2^r+1$.

The stronger real recurrence need not persist once an ultrametric factor is
present, even in the abelian quotient-group subclass. On the solenoid
\[
 X=(\mathbb R\times\mathbb Q_2)/\operatorname{diag}\mathbb Z[1/2],
 \qquad \alpha=(3/5,0)+\operatorname{diag}\mathbb Z[1/2],
\]
the first strict best returns are $1,3,6,7$. Hence $q_4=7<3+6$,
contradicting the real two-step recurrence. Nevertheless
Corollary~\ref{cor:mixed-groups} gives $g_N\leq3$. This is the adelic torus
of Das--Haynes for $\mathcal P=\{2\}$.

A packing recurrence does survive at the full homogeneous-space level.
\begin{proposition}[Effective packing recurrence on homogeneous spaces]
\label{prop:weak}
For $X=(V\times G)/\Lambda$ and the orbit of any $a\in V\times G$,
\[
    q_{n+P(\overline V)+1}>q_n+q_{n+1}
\]
whenever the left-hand side is defined. If $V$ is an inner-product space,
then
\[
    q_{n+P(E_\Lambda)+1}>q_n+q_{n+1}.
\]
\end{proposition}

For an inner-product norm this gives the index
$P(E_\Lambda)+1=\sigma_{m_\Lambda}^{>}+1$ when $m_\Lambda\geq1$. In effective dimension one this already yields the sharp three-step recurrence.
In effective dimension two, an additional geometric argument improves the
index from six to five. We prove it in the abelian model and then use
Theorem~\ref{thm:abelianization} to transfer it to every homogeneous space.

\begin{theorem}[Low-dimensional recurrence]\label{thm:lowdim}
Let $X=(V\times G)/\Lambda$ with $V$ an inner-product space.
\begin{enumerate}
\item If $m_\Lambda=1$, then
\[
    q_{n+3}>q_n+q_{n+1}
\]
whenever the left-hand side is defined.
\item If $m_\Lambda=2$, then
\[
    q_{n+5}>q_n+q_{n+1}
\]
whenever the left-hand side is defined.
\end{enumerate}
Both indices are best possible even within the abelian mixed-quotient subclass:
the first cannot be replaced by $2$, and the second cannot be replaced by
$4$, even if the desired conclusion is weakened to a non-strict inequality.
\end{theorem}

No normality or normalizer hypothesis is needed. The preserved effective factor in
Theorem~\ref{thm:abelianization} is essential for this transfer.

Theorem~\ref{thm:growth} turns Theorem~\ref{thm:lowdim} into $g_N\leq4$
in effective dimension one and $g_N\leq6$ in effective dimension two.
These are consequences of the recurrence rather than improvements of the
direct distance theorem: the packing bounds are already $3$ and $6$,
respectively.

The preceding record can sometimes improve Proposition~\ref{prop:weak}.
The next criterion refers to a representative in the given homogeneous
presentation; its proof is direct and requires no abelianization.

\begin{proposition}[A criterion for the stronger packing recurrence]
\label{prop:archimedean}
Suppose that $V$ is $(\mathbb R^d,\|\cdot\|_\infty)$ or a $d$-dimensional
inner-product space, with $d\geq1$. Let $q_{n+P(V)}$ be defined. If the coset
$a^{q_n}\Lambda$ has a representative $e_0=(v_0,g_0)$ such that
\begin{equation}\label{eq:archimedean-condition}
 D(e_H,e_0)=r_n,\qquad |g_0|_G<r_n,
\end{equation}
where $e_H=(0,e_G)$, then
\[
 q_{n+P(V)}\geq q_n+q_{n+1}.
\]
\end{proposition}

Table~\ref{tab:comparison} summarizes the relation with the principal earlier
results. The table distinguishes the packing bounds from the stronger purely
real distance and recurrence results. The mixed statements apply equally
to abelian quotients and arbitrary homogeneous presentations by
Theorem~\ref{thm:abelianization}.

\begin{table}[!ht]
\centering
\small
\setlength{\tabcolsep}{4pt}
\renewcommand{\arraystretch}{1.16}
\begin{tabularx}{\textwidth}{@{}L{.20\textwidth}L{.20\textwidth}L{.17\textwidth}Y@{}}
\toprule
Work & Setting & Statement & Specialization / comparison\\
\midrule
Chevallier \cite{Chevallier1996}
& $\mathbb R^2/\mathbb Z^2$, maximum norm
& $g_N\leq5$
& Corollary~\ref{cor:mixed-groups} with $U=\{e\}$,
$V=(\mathbb R^2,\|\cdot\|_\infty)$ and $\Lambda=\mathbb Z^2$.\\[3pt]

Haynes--Ramirez \cite{HaynesRamirez2021}
& Real $d$-tori, maximum norm
& $g_N\leq2^d+1$
& Corollary~\ref{cor:mixed-groups} with $U=\{e\}$ and a full-rank
lattice $\Lambda$.\\[3pt]

Das--Haynes \cite{DasHaynes2023}
& Adelic tori with their quotient metric
& $g_N\leq3$
& Corollary~\ref{cor:mixed-groups} with $V=\mathbb R$, the restricted
adelic ultrametric factor described above and the diagonal subgroup.\\[3pt]

Haynes--Marklof \cite{HaynesMarklof2022}
& Real Euclidean tori
& $5$ if $d=2$;\newline $K_d+1$ if $d\geq3$
& For $d\geq3$, Corollary~\ref{cor:mixed-groups} and the Euclidean packing
constant give $g_N\leq\sigma_d^{>}+1\leq K_d+1$.
For $d=2$, Proposition~\ref{prop:shulga-real} gives $g_N\leq5$.\\[3pt]

Shulga: angular distance bound \cite{Shulga2026}
& Real inner-product tori
& $g_N\leq\sigma_d^{>}+1$
& Corollary~\ref{cor:mixed-groups} with $U=\{e\}$ and a full-rank lattice.\\[3pt]

Shulga: distance theorem \cite{Shulga2026}
& Real inner-product tori
& $g_N\leq2^d+1$
& Proposition~\ref{prop:shulga-real} with $r=d$.\\[3pt]

Lagarias; Romanov; Shulga
\cite{Lagarias1982,Romanov2006,Shulga2026}
& Real tori: maximum norm; Euclidean $d=2$; inner-product norms
& $q_{n+2^d}\geq q_n+q_{n+1}$
& Lagarias is recovered from Proposition~\ref{prop:archimedean} in the real
maximum-norm case. Romanov and Shulga follow from
Proposition~\ref{prop:shulga-real}. In abelian mixed quotients the same index
can fail; Theorem~\ref{thm:lowdim} gives the sharp one-index weakening in
effective dimensions one and two.\\[3pt]

Chevallier; Shutov \cite{ChevallierApproximations1996,Shutov2024}
& Best returns on real tori
& Record counts and growth criteria
& Lemma~\ref{lem:running} and Theorem~\ref{thm:growth} hold for arbitrary
isometric embeddings, hence in particular for these toral translations.\\
\bottomrule
\end{tabularx}
\caption{Earlier distance and denominator-growth results as special cases
or comparisons. Packing bounds follow from Corollary~\ref{cor:mixed-groups};
the sharper purely real bounds use Proposition~\ref{prop:shulga-real}.}
\label{tab:comparison}
\end{table}

All proofs follow in Section~\ref{sec:proofs}. We first construct the
abelian model and verify preservation of the real data. We then establish
the packing constants, remove the continuous real part of the stabilizer,
and prove separation directly on the homogeneous space. The record-count
identity yields the distance theorem and the growth criterion. The later
subsections treat the purely real projection argument, the sharper mixed
recurrences, and the counterexamples. The planar proof takes place in an
abelian model and transfers back with the effective dimension unchanged.

\clearpage
\section{Proofs}\label{sec:proofs}

\subsection{Proof of the abelianization principle}
\label{subsec:abelianization-proof}

Fix $a=(v,h)$ and write
\[
 \Lambda_V=\{w\in V:(w,g)\in\Lambda\text{ for some }g\in G\},
 \qquad \mathcal F_w=\{g\in G:(w,g)\in\Lambda\}.
\]
Each fibre $\mathcal F_w$ is nonempty for $w\in\Lambda_V$, and the subgroup
property gives
\[
 e_G\in\mathcal F_0,\qquad
 \mathcal F_w^{-1}=\mathcal F_{-w},\qquad
 \mathcal F_w\mathcal F_{w'}\subset\mathcal F_{w+w'}.
\]
Define on the abelian group $A=\Lambda_V\times\mathbb Z$
\begin{equation}\label{eq:Phi}
 \Phi(w,q)=\inf_{g\in\mathcal F_w}d_G(h^q,g).
\end{equation}
The return function is therefore
\begin{equation}\label{eq:rho-Phi}
 \rho(q)=\inf_{w\in\Lambda_V}
       \max\{\|qv-w\|_V,\Phi(w,q)\}\qquad(q\in\mathbb Z).
\end{equation}

\begin{lemma}\label{lem:Phi-ultrametric}
The function $\Phi$ is an ultrametric pseudo-norm on $A$:
\[
 \Phi(0,0)=0,\qquad \Phi(-x)=\Phi(x),\qquad
 \Phi(x+y)\leq\max\{\Phi(x),\Phi(y)\}\qquad(x,y\in A).
\]
\end{lemma}

\begin{proof}
The first assertion follows from $e_G\in\mathcal F_0$. Inversion is an
isometry for a bi-invariant metric, so $\mathcal F_w^{-1}=\mathcal F_{-w}$
gives symmetry. For $g\in\mathcal F_w$ and $g'\in\mathcal F_{w'}$,
bi-invariance and the ultrametric inequality give
\begin{align*}
 d_G(h^{q+q'},gg')
 &\leq\max\{d_G(h^qh^{q'},h^qg'),d_G(h^qg',gg')\}\\
 &=\max\{d_G(h^{q'},g'),d_G(h^q,g)\}.
\end{align*}
Given $\varepsilon>0$, choose $g,g'$ so that the last two distances are
less than $\Phi(w',q')+\varepsilon$ and $\Phi(w,q)+\varepsilon$,
respectively. Since $gg'\in\mathcal F_{w+w'}$, letting
$\varepsilon\downarrow0$ proves the strong triangle inequality.
Neither commutativity of $G$ nor normalization of $\Lambda$ by $a$ is used.
\end{proof}

\begin{proof}[Proof of Theorem~\ref{thm:abelianization}]
Let $N_0=\{x\in A:\Phi(x)=0\}$. By Lemma~\ref{lem:Phi-ultrametric},
$N_0$ is a subgroup and $\Phi$ is constant on its cosets. Thus
\[
 U=A/N_0,\qquad |[x]|_U=\Phi(x),\qquad
 d_U([x],[y])=\Phi(x-y)
\]
defines an abelian group with a translation-invariant ultrametric.
No topology on $A$ is imposed in advance: $U$ carries the metric topology
just defined. Set $u=[(0,1)]$ and let
\[
 L=\{(w,[(-w,0)]):w\in\Lambda_V\}\leq V\times U,
 \qquad \Lambda_{\mathrm{ab}}=\overline L,
\]
where the closure is in the maximum product metric. The set $L$ is the
image of a homomorphism, so its closure is a closed subgroup. For
$w\in\Lambda_V$,
\[
 qu-[(-w,0)]=[(w,q)],
\]
and hence
\[
 D_{\mathrm{ab}}\bigl((qv,qu),(w,[(-w,0)])\bigr)
       =\max\{\|qv-w\|_V,\Phi(w,q)\}.
\]
Taking the infimum and using \eqref{eq:rho-Phi} gives
\[
 \operatorname{dist}_{D_{\mathrm{ab}}}((qv,qu),L)=\rho(q).
\]
Distance to a set equals distance to its closure. Consequently, with
$\alpha=(v,u)+\Lambda_{\mathrm{ab}}$,
\[
 d_{X_{\mathrm{ab}}}(q\alpha,0)=\rho(q)\qquad(q\in\mathbb Z).
\]
The isometric actions in both spaces now give
\eqref{eq:abelian-orbit-isometry}. In particular, zero distances and all
strict records agree, so the assertions about orbit sets and their data
follow even for periodic orbits or terminating record sequences.

We verify a slightly stronger equality than the preservation of $W$:
\begin{equation}\label{eq:abelian-real-kernel}
 (z,0)\in\Lambda_{\mathrm{ab}}
       \quad\Longleftrightarrow\quad (z,e_G)\in\Lambda
       \qquad(z\in V).
\end{equation}
For the forward implication choose $w_j\in\Lambda_V$ with
$w_j\to z$ and $[(-w_j,0)]\to0$ in $U$. Thus
$\Phi(w_j,0)=\Phi(-w_j,0)\to0$. Choose $g_j\in\mathcal F_{w_j}$ with
\[
 |g_j|_G<\Phi(w_j,0)+1/j.
\]
Then $(w_j,g_j)\in\Lambda$ tends to $(z,e_G)$, which belongs to
$\Lambda$ by closedness. Conversely, $(z,e_G)\in\Lambda$ implies
$\Phi(z,0)=\Phi(-z,0)=0$, so $(z,0)\in L$.
Applying \eqref{eq:abelian-real-kernel} to $z=tw$ for every real $t$
gives $W_{\mathrm{ab}}=W$ and the equality of the reduced normed factors.

Finally,
\[
 \Lambda_V\subset\pi_V(\Lambda_{\mathrm{ab}})
          \subset\overline{\Lambda_V}.
\]
After projection to $V/W$, these sets have the same real linear span,
because linear subspaces of a finite-dimensional real vector space are
closed. For an inner-product norm this is precisely
$E_{\Lambda_{\mathrm{ab}}}=E_\Lambda$, and equality of their dimensions
follows.
\end{proof}

\begin{remark}[Invariance needed for abelianization]
\label{rem:weaker-invariance}
The construction also works if $d_G$ is left invariant and is right
invariant only under $\Lambda_G=\pi_G(\Lambda)$. These assumptions make
\eqref{eq:quotient} a metric with an isometric left action. For
$g\in\mathcal F_w$, left invariance and right invariance under $g$ give
\[
 d_G(h^{-q},g^{-1})=d_G(e_G,h^qg^{-1})=d_G(g,h^q),
\]
which supplies the symmetry used above. The product estimate only uses
right invariance under $g'\in\mathcal F_{w'}$. Thus the same proof gives
the abelian model under these weaker assumptions, and all conclusions
depending only on the orbit metric and the preserved real data transfer
to this setting as well. We retain bi-invariance in the main statements
for a uniform ambient metric framework.
\end{remark}

\subsection{The geometric constants}

\begin{lemma}\label{lem:packing-values}
For every $d$-dimensional normed space $V$ with $d\geq1$,
$P(V)\leq3^d-1$.
Moreover,
\[
 P(\R^d,\|\cdot\|_\infty)=2^d,
 \qquad
 P(\R^d,\|\cdot\|_2)=\sk.
\]
\end{lemma}

\begin{proof}
Let $E$ be a finite configuration in \eqref{eq:packing}, and put
$r=\max_{v\in E}\|v\|_V<1$. The open balls of radius $1/2$ centred at
the points of $E$ are disjoint. They all lie in the ball of radius
$r+1/2$ centred at the origin. Comparing $d$-dimensional volumes gives
\[
 \card E\leq(2r+1)^d<3^d.
\]
An infinite admissible set would contain arbitrarily large finite
admissible subsets, so it is impossible. The set of attainable
cardinalities is a non-empty bounded set of integers and has a maximum.

For the maximum norm, partition $(-1,1)^d$ into $2^d$ regions according
to the signs of the coordinates, assigning zero to the non-negative
side. Two points in the same region have maximum-norm distance strictly
less than $1$. Therefore every admissible set has at most $2^d$ points.
The vertices of $\{-1/2,1/2\}^d$ form an admissible set of that size.
An invertible linear isometry preserves the packing number, proving the
assertion for linear images of the cube as well.

For the Euclidean norm, let $v_1,\ldots,v_m$ be an admissible set.
The case $m=1$ is immediate. If $m\geq2$, no $v_i$ is zero, because
its distance to every other point would be less than $1$. Write
\[
 v_i=a_i\xi_i,\qquad 0<a_i<1,\quad\|\xi_i\|_2=1.
\]
For a pair of indices, interchange them if necessary so that $a_i\leq a_j$.
If $\langle\xi_i,\xi_j\rangle\geq1/2$, then
\[
 \|v_i-v_j\|_2^2
 \leq a_i^2+a_j^2-a_i a_j
 =a_j^2+a_i(a_i-a_j)
 \leq a_j^2<1,
\]
contrary to admissibility. Thus the radial projections form a strong
kissing configuration, and $m\leq\sk$.

Conversely, a configuration $C\subset S^{d-1}$ counted by $\sk$ has
pairwise distances strictly greater than $1$. Since it is finite,
there is $t<1$, sufficiently close to $1$, such that all pairwise
distances in $tC$ remain at least $1$. Hence $tC$ is admissible in the
open unit ball and $P(\R^d,\|\cdot\|_2)\geq\sk$.
\end{proof}

The values $\sigma_1^{>}=2$ and $\sigma_2^{>}=5$ follow respectively from
the two-point sphere and the sum of angles around a circle. A regular
pentagon gives the lower bound in dimension two. In dimension three,
the ordinary kissing number is $12$; see \cite{HaynesMarklof2022}.
The twelve vertices of a regular icosahedron, normalized to lie on the
unit sphere, have maximal inner product $1/\sqrt5<1/2$. They therefore
give $\sigma_3^{>}=12$.

\begin{remark}
The open ball in Definition~\ref{def:packing} is essential. A closed
unit ball even permits its centre together with an ordinary spherical
kissing configuration. It is not legitimate to replace the open-ball
packing number by a closed-ball version without changing the constant.
\end{remark}

\subsection{Separation of improvements}

\begin{lemma}[Removing the continuous part]\label{lem:continuous-part}
For every closed subgroup $\Lambda\leq V\times G$, its identity component has
the form
\[
    \Lambda^\circ=W\times\{e_G\}
\]
with $W$ a real vector subspace of $V$. Equip $\overline V=V/W$ with the
quotient norm and let $\overline\Lambda$ be the image of $\Lambda$ in
$\overline V\times G$. Then $\overline\Lambda$ is closed and the natural map
\[
 (V\times G)/\Lambda\longrightarrow
 (\overline V\times G)/\overline\Lambda
\]
is an isometry of homogeneous spaces. If $V$ is an inner-product space, one
may identify $\overline V$ with $W^\perp$ and $\overline\Lambda$ with
$\Lambda\cap(W^\perp\times G)$.
\end{lemma}

\begin{proof}
The projection of the connected group $\Lambda^\circ$ to the ultrametric group
$G$ is connected. Every connected subset of an ultrametric space is a
singleton, so this projection is $\{e_G\}$ and
$\Lambda^\circ\subset V\times\{e_G\}$. A connected closed subgroup of the
additive group $V$ is a real vector subspace, giving
$\Lambda^\circ=W\times\{e_G\}$.

Let
\[
    \Pi:V\times G\longrightarrow (V/W)\times G
\]
be the quotient homomorphism. Its kernel $W\times\{e_G\}$ lies in $\Lambda$,
so $\Pi^{-1}(\overline\Lambda)=\Lambda$. Since $\Pi$ is a quotient map,
$\overline\Lambda$ is closed. For $x=(v,g)$ and $y=(w,h)$, the homogeneous
quotient distance is
\[
 d_X(x\Lambda,y\Lambda)
 =\inf_{(z,k)\in\Lambda}
   \max\{\|v-w-z\|_V,d_G(g,hk)\}.
\]
Taking first the infimum over the $W$-component gives
\[
 \inf_{u\in W}\max\{\|v-w-z-u\|_V,d_G(g,hk)\}
 =\max\{\|v-w-z+W\|_{V/W},d_G(g,hk)\},
\]
and then taking the infimum over $\Lambda$ proves the asserted isometry.
For an inner-product norm, $\|v+W\|_{V/W}=\|\pi(v)\|$. Multiplying an element
$(v,g)\in\Lambda$ by $(-w,e_G)\in W\times\{e_G\}\subset\Lambda$ removes its
$W$-component, which identifies $\overline\Lambda$ with
$\Lambda\cap(W^\perp\times G)$.
\end{proof}

We prove a slightly more flexible version of the dyadic argument, which
will also give Proposition~\ref{prop:weak}.

\begin{lemma}[Separation in a short interval]\label{lem:separation}
Let $Q\geq1$ with $R(Q)>0$. Suppose that $t_1,\ldots,t_m$ are distinct
positive integers satisfying
\[
 |t_i-t_j|\leq Q\quad(i\ne j),\qquad
 \rho(t_i)<R(Q)\quad(1\leq i\leq m).
\]
Then $m\leq P(V)$. Under either compatible splitting hypothesis in
Remark~\ref{rem:effective-factor}, one has the sharper bound $m\leq P(E)$.
\end{lemma}

\begin{proof}
Put $r=R(Q)$ and write $e_H=(0,e_G)$. For each $i$, the strict inequality
$\rho(t_i)<r$ and the definition of the homogeneous quotient metric allow us
to choose a representative
\[
 e_i=(v_i,g_i)\in a^{t_i}\Lambda,
 \qquad D(e_H,e_i)<r.
\]
Thus
\[
    \|v_i\|_V<r,\qquad |g_i|_G<r.
\]
For $i\ne j$, the definition of $r$ and the isometric left action give
\[
 r\leq\rho(|t_i-t_j|)
 =d_X(a^{t_i}\Lambda,a^{t_j}\Lambda)
 =d_X(e_i\Lambda,e_j\Lambda).
\]
The quotient projection is $1$-Lipschitz, hence
\[
 d_X(e_i\Lambda,e_j\Lambda)\leq D(e_i,e_j)
 =\max\{\|v_i-v_j\|_V,d_G(g_i,g_j)\}.
\]
By ultrametricity,
\[
 d_G(g_i,g_j)\leq\max\{|g_i|_G,|g_j|_G\}<r.
\]
Consequently $\|v_i-v_j\|_V\geq r$. The points $v_i/r$ lie in the open
unit ball of $V$ and have pairwise distances at least $1$, so $m\leq P(V)$.
Notice that no product or quotient-group operation on $X$ has been used.

For the refinement, assume $m\geq1$ and write the real component of $a$ as
$a_E+b$, with $a_E\in E$ and $b\in F$. Because
$\overline\Lambda\subset E\times G$, every chosen representative has real
component
\[
    v_i=x_i+t_i b,\qquad x_i\in E.
\]
Put $c=\|b\|_F$. Every $t_i>Q$, since $\rho(t_i)<R(Q)$.

\emph{Orthogonal inner-product decomposition.}
The lift inequalities and the separation just proved give
\[
 \|x_i\|^2+t_i^2c^2<r^2,\qquad
 \|x_i-x_j\|^2+(t_i-t_j)^2c^2\geq r^2.
\]
In particular $r>Qc$, and
$s=\sqrt{r^2-Q^2c^2}>0$. Since $t_i>Q$ and $|t_i-t_j|\leq Q$,
\begin{equation}\label{eq:reduced-separation}
 \|x_i\|<s,\qquad \|x_i-x_j\|\geq s\quad(i\ne j).
\end{equation}
Thus $x_i/s$ form an admissible packing in $E$, proving $m\leq P(E)$.

\emph{Maximum-product decomposition.}
Here $\|x_i\|_E<r$ and $t_i c<r$. Hence
$\|(t_i-t_j)b\|_F\leq Qc<r$, whereas $\|v_i-v_j\|_V\geq r$.
The maximum-product norm forces $\|x_i-x_j\|_E\geq r$, and again
$m\leq P(E)$.
\end{proof}

\begin{corollary}[Dyadic improvement bound]\label{cor:dyadic}
If $Q\geq1$ and $R(Q)>0$, then
\begin{equation}\label{eq:dyadic}
 \card\{q\in\Z:Q<q\leq2Q,\ \rho(q)<R(Q)\}\leq P(V).
\end{equation}
In particular, $R$ has at most $P(V)$ strict decreases on $(Q,2Q]$.
Under either compatible splitting hypothesis in Remark~\ref{rem:effective-factor},
both assertions hold with $P(E)$ in place of $P(V)$.
\end{corollary}

\begin{proof}
Any two integers in $(Q,2Q]$ differ by less than $Q$, so
Lemma~\ref{lem:separation} applies. If $R(t)<R(t-1)$ with $t>Q$,
then $\rho(t)=R(t)<R(Q)$; hence every strict decrease is counted in
\eqref{eq:dyadic}.
\end{proof}

\subsection{Running minima and the distance theorem}

\begin{lemma}[Running-minimum and record-count identities]
\label{lem:running}
Let $T$ be an isometric embedding of a metric space into itself and suppose that
$x_0,Tx_0,\ldots,T^{N-1}x_0$ are distinct. Put $Q=\lfloor N/2\rfloor$.
Then, for $0\leq j\leq N-1$,
\begin{equation}\label{eq:running}
 \delta_N(T^j x_0)=R\bigl(\max\{j,N-1-j\}\bigr).
\end{equation}
Moreover,
\begin{equation}\label{eq:exact-count}
 g_N
 =\card\{R(t):Q\leq t\leq N-1,\ t\in\mathbb Z\}
 =1+\card\{n:Q<q_n\leq N-1\}.
\end{equation}
\end{lemma}

\begin{proof}
For $0\leq j<k$, the map $T^j$ preserves distances, so
\[
 d(T^j x_0,T^k x_0)
   =d(T^j x_0,T^j(T^{k-j}x_0))
   =d(x_0,T^{k-j}x_0)=\rho(k-j).
\]
Symmetry gives the corresponding formula for $j>k$. No negative iterate
or inverse map is used.
The absolute values of the index differences from $j$ to the other indices
fill the integer interval from $1$ to $\max\{j,N-1-j\}$, proving
\eqref{eq:running}. As $j$ varies, $\max\{j,N-1-j\}$ takes every integer
value from $\lfloor N/2\rfloor$ to $N-1$. This proves the first equality in
\eqref{eq:exact-count}. There is one initial value $R(Q)$, and each strict
record time after $Q$ and at most $N-1$ produces exactly one further value,
which proves the second equality.
\end{proof}

This is the record-count mechanism behind Chevallier's lemma
\cite{Chevallier1996,ChevallierApproximations1996,Shutov2024}, written with $N$ orbit points instead
of indices $0,\ldots,N$. The identity explicitly handles the midpoint
and avoids a separate parity convention for $N$.

\begin{proof}[Proof of Theorem~\ref{thm:packing} and
Corollaries~\ref{cor:euclidean} and \ref{cor:sup}]
First suppose that the first $N$ orbit points are distinct. With
$Q=\lfloor N/2\rfloor$ we have $R(Q)>0$ and $N-1\leq2Q$. Every record time
counted in \eqref{eq:exact-count} lies in $(Q,2Q]$ and satisfies
$\rho(q_n)<R(Q)$. Corollary~\ref{cor:dyadic} therefore gives
\[
    g_N(a,\Lambda)\leq P(V)+1.
\]
By Lemma~\ref{lem:continuous-part}, the same metric homogeneous space has the
reduced presentation
\[
    X=((\overline V\times G)/\overline\Lambda).
\]
Applying the same argument there gives
\[
    g_N(a,\Lambda)\leq P(\overline V)+1.
\]

If $V$ is an inner-product space, identify $\overline V$ with $W^\perp$.
Then $\overline\Lambda\subset E_\Lambda\times G$ and
\[
    \overline V=E_\Lambda\oplus E_\Lambda^\perp
\]
orthogonally. The refined dyadic separation bound gives
$g_N(a,\Lambda)\leq P(E_\Lambda)+1$.

If the orbit segment has repetitions, the point $\Lambda$ is periodic under
the bijective isometry $L_a$. If its least period is $h<N$, the orbit set is
the entire $h$-cycle. The cyclic isometry acts transitively on this set, so
all nearest-neighbour distances are equal when $h\geq2$; for $h=1$ our
convention gives $g_N=0$. Thus the theorem holds in all cases.

Corollary~\ref{cor:euclidean} follows from Lemma~\ref{lem:packing-values} in
$E_\Lambda$. For Corollary~\ref{cor:sup}, the ambient cubical bound follows
from $P(V)=2^d$. In the reduced cubical coordinates, let $E$ be the coordinate
subspace spanned by the $k$ coordinates occurring in the real projection of
$\overline\Lambda$ and let $F$ be the complementary coordinate subspace.
Then $\overline\Lambda\subset E\times G$, the norm is a maximum-product norm
on $E\oplus F$, and $P(E)=2^k$. Remark~\ref{rem:effective-factor} gives the
sharper bound.
\end{proof}

\subsection{Growth criteria and the packing recurrence}
\label{subsec:growth-proofs}

\begin{proof}[Proof of Theorem~\ref{thm:growth}]
If two iterates $T^i x_0,T^j x_0$ coincide with $0\leq i<j<N$,
distance preservation by $T^i$ gives
$d(x_0,T^{j-i}x_0)=0$. Thus $x_0$ is periodic. For its least period
$h<N$, the orbit set is the whole $h$-cycle and the restriction of $T$
is a cyclic isometry of that set, so $g_N\leq1$ (and $g_N=0$ if $h=1$).
This uses no surjectivity on the ambient space. We may therefore assume
that the first $N$ iterates are distinct. Put $Q=\lfloor N/2\rfloor$. If more than $M$ record times
belonged to $(Q,N-1]$, then for some $k$,
\[
 Q<q_k<q_{k+1}<\cdots<q_{k+M}\leq N-1\leq2Q.
\]
This contradicts \eqref{eq:doubling}, since
$q_{k+M}\geq2q_k>2Q$.
The record-count identity \eqref{eq:exact-count} therefore gives
$g_N\leq M+1$. Finally, \eqref{eq:recurrence-M} implies
\eqref{eq:doubling}, because $q_{n+1}>q_n$.
\end{proof}

\begin{lemma}[Recurrences pass to subsequences]\label{lem:subsequence-recurrence}
Let $p_1<p_2<\cdots$ be a finite or infinite increasing sequence and
suppose that, for some $M\geq1$,
\[
 p_{j+M}\geq p_j+p_{j+1}
\]
whenever $p_{j+M}$ is defined. Then every subsequence
$q_n=p_{k_n}$ satisfies
\[
 q_{n+M}\geq q_n+q_{n+1}
\]
whenever $q_{n+M}$ is defined.
\end{lemma}

\begin{proof}
Put $j=k_{n+1}-1$. Since the $k_i$ are strictly increasing,
\[
 j\geq k_n,\qquad j+M\leq k_{n+M}.
\]
Hence
\[
 q_{n+M}=p_{k_{n+M}}
 \geq p_{j+M}
 \geq p_j+p_{j+1}
 \geq p_{k_n}+p_{k_{n+1}}
 =q_n+q_{n+1}.
\]
\end{proof}

\begin{proof}[Proof of Proposition~\ref{prop:shulga-real}]
Let $W=\Lambda^\circ$. Lemma~\ref{lem:continuous-part}, with $G=\{e_G\}$,
identifies $V/\Lambda$ isometrically with $W^\perp/\mathcal L$, where
\[
 \mathcal L=\Lambda\cap W^\perp.
\]
The decomposition $\Lambda=W\oplus\mathcal L$ shows that
$\mathcal L$ is a discrete free abelian subgroup with
\[
 \operatorname{rank}\mathcal L
 =\operatorname{rank}(\Lambda/\Lambda^\circ)=r.
\]
We may therefore replace $(V,\Lambda)$ by $(W^\perp,\mathcal L)$.

Let
\[
 E=\operatorname{span}_{\mathbb R}\mathcal L,
 \qquad F=W^\perp\cap E^\perp.
\]
Then $W^\perp=E\oplus F$, $\dim E=r$, and $\mathcal L$ is a full-rank
lattice in $E$. Write the projected rotation vector as
\[
 \alpha=a+b,\qquad a\in E,\quad b\in F,
\]
and define
\[
 \delta(q)=\dist(qa,\mathcal L).
\]
Orthogonality gives
\begin{equation}\label{eq:real-projection-distance}
 \rho(q)^2=\delta(q)^2+q^2\|b\|^2.
\end{equation}

Every strict best-return denominator for $\rho$ is a strict
best-approximation denominator for $\delta$. Indeed, if $q$ is a strict
record for $\rho$ and some $1\leq p<q$ satisfied
$\delta(p)\leq\delta(q)$, then for $b\ne0$,
\[
 \rho(p)^2
 \leq\delta(q)^2+p^2\|b\|^2
 <\delta(q)^2+q^2\|b\|^2
 =\rho(q)^2,
\]
a contradiction. If $b=0$, the claim is immediate from
$\rho=\delta$.

Assume first that $r\geq1$, and let
\[
 p_1<p_2<\cdots
\]
be the strict best-approximation denominators for the compact flat torus
$E/\mathcal L$. The preceding observation says that the strict
best-return denominators $q_n$ for $V/\Lambda$ form a subsequence of
the $p_j$. Shulga's Theorem~1.4 \cite{Shulga2026} gives, in particular,
\[
 p_{j+2^r}\geq p_j+p_{j+1}.
\]
Lemma~\ref{lem:subsequence-recurrence} therefore yields
\eqref{eq:shulga-real}. The distance bound follows from
Theorem~\ref{thm:growth}.

If $r=0$, then $\mathcal L=\{0\}$ and
\[
 \rho(q)=q\|b\|.
\]
This is either identically zero or strictly increasing, so no strict
record occurs after $q_1=1$. The record-count identity
\eqref{eq:exact-count}, together with the finite-orbit case already
handled in the proof of Theorem~\ref{thm:packing}, gives $g_N\leq1$.
\end{proof}

\begin{proof}[Proof of Proposition~\ref{prop:weak}]
First work in the ambient presentation and put $p=P(V)$. If
$q_{n+p+1}\leq q_n+q_{n+1}$, then the $p+1$ distinct times
\[
    q_{n+1},\ldots,q_{n+p+1}
\]
have return distances strictly below $r_n=R(q_n)>0$, and any two differ by
at most $q_n$. Lemma~\ref{lem:separation}, with $Q=q_n$, would give
$p+1\leq p$, a contradiction. This argument only uses the homogeneous-space
separation lemma.

Applying the same argument to the reduced presentation supplied by
Lemma~\ref{lem:continuous-part} gives the recurrence with
$p=P(\overline V)$. If $V$ is an inner-product space, then in the reduced
presentation
$\overline\Lambda\subset E_\Lambda\times G$ and
$\overline V=E_\Lambda\oplus E_\Lambda^\perp$ orthogonally. The refined form
of Lemma~\ref{lem:separation} gives the same contradiction with
$p=P(E_\Lambda)$.
\end{proof}

The same separation and record-count proofs also establish
Remark~\ref{rem:purely-ultrametric}: for $V=\{0\}$ an admissible set
in its open unit ball $\{0\}$ has at most one point, and no volume
argument is needed. The displayed distances on $\Z_2$ prove sharpness.

Since $P(E_\Lambda)=2$ when $m_\Lambda=1$,
Proposition~\ref{prop:weak} proves the one-dimensional assertion of
Theorem~\ref{thm:lowdim}.

\subsection{Using the preceding record}
\label{subsec:archimedean-proof}

\begin{proof}[Proof of Proposition~\ref{prop:archimedean}]
Write $r=r_n$ and $M=P(V)$, and suppose that
\[
 q_{n+M}<q_n+q_{n+1}.
\]
By definition of the next strict record,
\begin{equation}\label{eq:before-next-record}
 1\leq m<q_{n+1}\quad\Longrightarrow\quad \rho(m)\geq r.
\end{equation}
Let $e_0=(v_0,g_0)\in a^{q_n}\Lambda$ be the representative in
\eqref{eq:archimedean-condition}. Then $\|v_0\|_V=r$ and $|g_0|_G<r$.
For $1\leq i\leq M$, choose representatives
\[
 e_i=(v_i,g_i)\in a^{q_{n+i}}\Lambda,\qquad D(e_H,e_i)<r.
\]
For every $0\leq i<j\leq M$,
\[
 0<q_{n+j}-q_{n+i}\leq q_{n+M}-q_n<q_{n+1}.
\]
By the isometric action,
\[
 d_X(e_i\Lambda,e_j\Lambda)=\rho(q_{n+j}-q_{n+i})\geq r.
\]
On the other hand, $d_X(e_i\Lambda,e_j\Lambda)\leq D(e_i,e_j)$ and
ultrametricity gives $d_G(g_i,g_j)<r$. Hence
\begin{equation}\label{eq:boundary-configuration}
 \|v_0\|_V=r,\qquad
 \|v_i\|_V<r\ (i\geq1),\qquad
 \|v_i-v_j\|_V\geq r\ (i\ne j).
\end{equation}
Thus the rest of the proof is purely real.

\emph{Maximum norm.}
Assign to $v_0$ the signs of its coordinates, assigning a positive sign to a
zero coordinate. For each $v_i$ with $i\geq1$, retain the sign of every
nonzero coordinate, and assign to a zero coordinate the sign opposite to that
of $v_0$ in that coordinate. There are $2^d$ patterns and $2^d+1$ vectors, so
two patterns coincide.

If the two vectors are interior vectors, their entries in each coordinate have
the same weak sign and absolute values less than $r$, so their maximum-norm
distance is less than $r$. If the pair is $v_0,v_i$, then in a coordinate
where $(v_0)_k\ne0$, the matching pattern forces $(v_i)_k\ne0$ with the same
sign, whence $|(v_i)_k-(v_0)_k|<r$. In a coordinate where $(v_0)_k=0$ the
same strict bound is immediate. Both cases contradict
\eqref{eq:boundary-configuration}.

\emph{Inner-product norm.}
Here $M=\sigma_d^{>}\geq2$. None of the interior vectors can be zero. Write
$v_i=a_i\xi_i$ with $a_i>0$ and $\|\xi_i\|_V=1$. Suppose that
$\langle\xi_i,\xi_j\rangle\geq1/2$ for some pair, ordered so that
$a_i\leq a_j$. Then
\[
 \|v_i-v_j\|_V^2
 \leq a_i^2+a_j^2-a_i a_j
 =a_j^2-a_i(a_j-a_i).
\]
If $a_j<r$, this is $<r^2$. If $a_j=r$, then $a_i<r=a_j$ because only
$v_0$ lies on the boundary, and the expression is again $<r^2$. This
contradicts \eqref{eq:boundary-configuration}. Thus
$\xi_0,\ldots,\xi_M$ would form a strong kissing configuration of
$\sigma_d^{>}+1$ points, impossible by definition.
\end{proof}

\subsection{Five points in a disc and the planar recurrence}
\label{subsec:planar}

We identify the Euclidean plane with $\mathbb C$. The main geometric
step concerns five points in the open unit disc with pairwise
distances at least one. We first establish an angular estimate.

\begin{lemma}[Angular estimate]\label{lem:angular}
Let $0<A,B,C\le1$, put
\[
 p=\arccos\frac A2,\qquad q=\arccos\frac B2,\qquad T=\frac{5\pi}3-p-q,
\]
and let $\alpha,\beta$ satisfy
\begin{equation}\label{eq:angular-domain}
 \alpha\ge p,\qquad\beta\ge q,\qquad\alpha+\beta\le T.
\end{equation}
Then
\begin{equation}\label{eq:angular-target}
 \bigl|Ae^{-i\alpha}+Be^{i\beta}-C\bigr|\le1.
\end{equation}
\end{lemma}

\begin{proof}
Note that $p,q\in[\pi/3,\pi/2)$, and that \eqref{eq:angular-domain} is
consistent only if $p+q\le T$, that is $p+q\le5\pi/6$; then
$T\le\pi$. For fixed $A,B,\alpha,\beta$ the square of the left-hand side of
\eqref{eq:angular-target} is a convex quadratic function of $C$, so it
suffices to treat $C=0$ and $C=1$.

\emph{The case $C=0$.} Put $x=\pi/2-p$, $y=\pi/2-q$, $s=x+y$ and
$\delta=x-y$, so that $A=2\sin x$, $B=2\sin y$ and $0<x,y\le\pi/6$. From
$p+q\le5\pi/6$ and $p,q\ge\pi/3$ we get $\pi/6\le s\le\pi/3$, and from
$x,y\le\pi/6$ we get
\begin{equation}\label{eq:delta-bound}
 |\delta|\le\frac\pi3-s.
\end{equation}
Since $S=\alpha+\beta\in[p+q,T]\subset[p+q,\pi]$ and the cosine decreases on
$[0,\pi]$,
\[
 |Ae^{-i\alpha}+Be^{i\beta}|^2=A^2+B^2+2AB\cos S\le A^2+B^2+2AB\cos(p+q).
\]
With $p+q=\pi-s$, the identities $\sin^2x+\sin^2y=1-\cos s\cos\delta$ and
$2\sin x\sin y=\cos\delta-\cos s$ give
\[
 A^2+B^2+2AB\cos(p+q)=4\bigl(1+\cos^2s-2\cos s\cos\delta\bigr).
\]
By \eqref{eq:delta-bound}, $\cos\delta\ge\cos(\pi/3-s)$, and
$2\cos s\cos(\pi/3-s)=\tfrac12+\cos(2s-\pi/3)$, whence
\[
 \tfrac14|Ae^{-i\alpha}+Be^{i\beta}|^2
 \le1+\cos^2s-2\cos s\cos\Bigl(\frac\pi3-s\Bigr)
 =1-\frac{\sqrt3}2\sin(2s)\le\frac14,
\]
because $2s\in[\pi/3,2\pi/3]$. This is \eqref{eq:angular-target} for $C=0$.

\emph{The case $C=1$.} Put $H(\alpha,\beta)=|Ae^{-i\alpha}+Be^{i\beta}-1|^2$
and let $D$ be the triangle defined by \eqref{eq:angular-domain}. Expanding,
\[
 H=A^2+B^2+1+2AB\cos(\alpha+\beta)-2A\cos\alpha-2B\cos\beta .
\]
For fixed $\alpha$,
\[
 \frac1{2B}\frac{\partial H}{\partial\beta}
 =\sin\beta-A\sin(\alpha+\beta)
 =(1-A\cos\alpha)\sin\beta-A\sin\alpha\cos\beta,
\]
with $1-A\cos\alpha>0$ and $A\sin\alpha\ge0$ on $D$. A function
$a\sin\beta-b\cos\beta$ with $a>0$, $b\ge0$ changes sign on $(0,\pi)$ at most
once, from negative to positive. Hence on every vertical section of $D$ the
maximum of $H$ is attained at an endpoint, and it suffices to maximize $H$ on
the edges $\beta=q$ and $\beta=T-\alpha$. On the edge $\beta=q$,
\[
 \frac1{2A}\frac{d}{d\alpha}H(\alpha,q)=(1-B\cos q)\sin\alpha-B\sin q\cos\alpha
\]
has the same form and again admits no interior local maximum. On the edge
$\beta=T-\alpha$,
\[
 \frac12\frac{d}{d\alpha}H(\alpha,T-\alpha)=A\sin\alpha-B\sin(T-\alpha),
\]
and for $0<T<\pi$ the ratio $\sin\alpha/\sin(T-\alpha)$ is strictly
increasing on $(0,T)$, so the derivative vanishes at most once and changes
sign from negative to positive; if $T=\pi$, the displayed half-derivative is $(A-B)\sin\alpha$, which has constant sign. Thus
the maximum of $H$ over $D$ is attained at a vertex $(p,q)$, $(p,T-p)$ or
$(T-q,q)$.

At $(p,q)$: since $A=2\cos p$ and $B=2\cos q$, we have $Ae^{-ip}=1+e^{-2ip}$
and $Be^{iq}=1+e^{2iq}$, so with $x,y,s,\delta$ as above
\[
 Ae^{-ip}+Be^{iq}-1=1-e^{2ix}-e^{-2iy}=1-2\cos s\,e^{i\delta},
\]
and $|1-2\cos s\,e^{i\delta}|^2=1+4\cos^2s-4\cos s\cos\delta\le1$ because
$|\delta|\le s$ and $\cos s>0$.

At $(p,T-p)$: $Ae^{-ip}-1=e^{-2ip}$, so
\begin{align*}
 H(p,T-p)&=1+B^2+2B\cos(T+p)=1+B\Bigl(2\cos q+2\cos\Bigl(\frac{5\pi}3-q\Bigr)\Bigr)\\
 &=1-2\sqrt3\,B\cos\Bigl(q-\frac{5\pi}6\Bigr)\le1,
\end{align*}
since $q-5\pi/6\in[-\pi/2,-\pi/3]$. The vertex $(T-q,q)$ is symmetric.
Hence $H\le1$ on $D$, which is \eqref{eq:angular-target} for $C=1$.
\end{proof}

\begin{lemma}[Five points in the open unit disc]\label{lem:five-point}
Let $z_1,\dots,z_5\in\R^2$ satisfy $\|z_i\|_2<1$ and $\|z_i-z_j\|_2\ge1$ for
$i\ne j$. Let $c$ be one of the points and let $a,b$ be its two neighbours in
the cyclic order of the directions from the origin. Then
\begin{equation}\label{eq:five-point-target}
 \|a+b-c\|_2\le\max\{\|a\|_2,\|b\|_2,\|c\|_2\}.
\end{equation}
\end{lemma}

\begin{proof}
No $z_i$ is zero, since otherwise its distance to every other point would be
less than $1$. By Lemma~\ref{lem:packing-values}, the smaller angle between any two directions exceeds $\pi/3$. Thus each successive angular gap exceeds $\pi/3$ and, since there are five, is less than $2\pi/3$. Rotate so that
$c=C>0$ is real, and write $a=Ae^{-i\alpha}$, $b=Be^{i\beta}$ with $A,B>0$,
where $\alpha,\beta>0$ are the angular gaps from $c$ to its two neighbours.
Put $p=\arccos(A/2)$ and $q=\arccos(B/2)$.

From $|a-c|\ge1$ we get $1\le A^2+C^2-2AC\cos\alpha$, hence
\[
 \cos\alpha\le\frac{A^2+C^2-1}{2AC}<\frac A2,
\]
the last inequality because $A^2+C^2-1-A^2C=(C-1)(C+1-A^2)<0$. Thus
$\alpha>p$, and likewise $\beta>q$. The complementary arc from $b$ to $a$
contains the two remaining points and consists of three angular gaps. The gap
adjacent to $b$ exceeds $q$, by the same computation with $b$ in the role of
$c$'s neighbour; the gap adjacent to $a$ exceeds $p$; and the middle gap
exceeds $\pi/3$, because two points of norm less than $1$ at distance at
least $1$ subtend an angle strictly greater than $60^\circ$
(Lemma~\ref{lem:packing-values}). Hence
\begin{equation}\label{eq:angular-budget}
 \alpha+\beta<2\pi-\Bigl(p+\frac\pi3+q\Bigr)=\frac{5\pi}3-p-q .
\end{equation}

Let $M=\max\{A,B,C\}<1$ and put $A'=A/M$, $B'=B/M$, $C'=C/M$, so that
$0<A',B',C'\le1$, and $p'=\arccos(A'/2)\le p$, $q'=\arccos(B'/2)\le q$.
Then $\alpha>p\ge p'$, $\beta>q\ge q'$ and, by \eqref{eq:angular-budget},
$\alpha+\beta<5\pi/3-p-q\le5\pi/3-p'-q'$. Lemma~\ref{lem:angular} applied to
$A',B',C'$ gives $|A'e^{-i\alpha}+B'e^{i\beta}-C'|\le1$; multiplying by $M$
yields \eqref{eq:five-point-target}.
\end{proof}

\begin{proof}[Proof of the planar recurrence in Theorem~\ref{thm:lowdim}]
Apply Theorem~\ref{thm:abelianization}. The strict return times and the
effective real factor are unchanged, so it suffices to work in the
abelian mixed quotient $(V\times U)/\Lambda_{\mathrm{ab}}$.
Lemma~\ref{lem:continuous-part} then removes the continuous real part.
For clarity, rename the resulting closed subgroup $\Lambda$ and the
generator $a=(v,u)$, and write the abelian groups additively. We have
an orthogonal splitting
\[
 V=E\oplus E^\perp,\qquad \dim E=2,\qquad \Lambda\subset E\times U.
\]
The metric is the maximum of the real norm and the ultrametric on $U$.

Suppose, to the contrary, that
\begin{equation}\label{eq:planar-contradiction}
 q_{n+5}\leq q_n+q_{n+1}.
\end{equation}
Put $r=r_n>0$ and write the real component of $a$ as $a_E+b$, with
$b\in E^\perp$. Put $c=\|b\|$ and $t_i=q_{n+i}$ for $1\leq i\leq5$.
The strict record property gives
\begin{equation}\label{eq:strict-record-property}
 1\leq m<q_k\quad\Longrightarrow\quad \rho(m)>r_k.
\end{equation}
In particular,
\[
 \eta=\min_{1\leq i\leq4}
       \bigl(R(q_{n+i}-1)-r_{n+i}\bigr)>0.
\]
Choose
\[
 0<\varepsilon<\min\{\eta,r_n-r_{n+1}\}.
\]
For $1\leq i\leq5$, choose a representative
\begin{equation}\label{eq:approximate-lifts}
 e_i=(v_i,u_i)\in t_i a+\Lambda,
 \qquad D(0,e_i)<r_{n+i}+\varepsilon<r.
\end{equation}
Because $\Lambda\subset E\times U$, the real coordinate has the form
\[
    v_i=x_i+t_i b,\qquad x_i\in E,
\]
and hence
\begin{equation}\label{eq:planar-projected-lifts}
 \|x_i\|^2+t_i^2c^2<(r_{n+i}+\varepsilon)^2.
\end{equation}

If $i<j$, then \eqref{eq:planar-contradiction} gives
\[
 0<q_{n+j}-q_{n+i}\leq q_n.
\]
Therefore
\[
 d_X(e_i+\Lambda,e_j+\Lambda)
 =\rho(q_{n+j}-q_{n+i})\geq r.
\]
Since the quotient projection is $1$-Lipschitz and
$d_U(u_i,u_j)<r$, we obtain $\|v_j-v_i\|\geq r$. As in
\eqref{eq:reduced-separation}, with $Q=q_n$, the positive number
$s=\sqrt{r^2-q_n^2c^2}$ satisfies
\[
 \|x_i\|<s,\qquad \|x_i-x_j\|\geq s\quad(i\ne j).
\]
The five vectors $x_i/s$ therefore satisfy Lemma~\ref{lem:five-point}.
Apply it with distinguished point $x_5/s$, and let $x_i/s,x_j/s$ be its two
angular neighbours, ordered so that $i<j<5$. Then
\[
 \|x_i+x_j-x_5\|
 \leq\max\{\|x_i\|,\|x_j\|,\|x_5\|\}.
\]
For $\ell\in\{i,j,5\}$, monotonicity of the record distances and
$t_\ell\geq t_i$ give
\[
 \|x_\ell\|^2<(r_{n+i}+\varepsilon)^2-t_i^2c^2.
\]

Set
\[
    m=t_i+t_j-t_5.
\]
The inequalities $t_i\geq q_{n+1}$, $t_j>q_n$ and
\eqref{eq:planar-contradiction} imply $m>0$, while $t_j<t_5$ gives $m<t_i$.
The abelian model allows us to combine the three representatives:
\[
    e_i+e_j-e_5\in(t_i+t_j-t_5)a+\Lambda=ma+\Lambda.
\]
Its real component is $v_i+v_j-v_5$, and
\begin{align*}
 \|v_i+v_j-v_5\|^2
 &=\|x_i+x_j-x_5\|^2+m^2c^2\\
 &<(r_{n+i}+\varepsilon)^2-(t_i^2-m^2)c^2\\
 &\leq(r_{n+i}+\varepsilon)^2.
\end{align*}
Its $U$-component is $u_i+u_j-u_5$, whose norm is less than
$r_{n+i}+\varepsilon$ by ultrametricity and \eqref{eq:approximate-lifts}. Hence
\[
    \rho(m)<r_{n+i}+\varepsilon.
\]
But $0<m<q_{n+i}$ implies
\[
 \rho(m)\geq R(q_{n+i}-1)
 \geq r_{n+i}+\eta
 >r_{n+i}+\varepsilon,
\]
a contradiction. Thus $q_{n+5}>q_n+q_{n+1}$.
\end{proof}

\subsection{Counterexamples to the stronger recurrence}
\label{subsec:examples}

We verify the two examples announced in the introduction. Together
with the preceding proofs they establish the sharpness assertion
in Theorem~\ref{thm:lowdim}.

\paragraph{One real dimension.}
Let
\[
 X=(\R\times\Q_2)/\operatorname{diag}\Z[1/2],
 \qquad
 \alpha=(3/5,0)+\operatorname{diag}\Z[1/2],
\]
with the maximum of the usual real and $2$-adic absolute values.
Then
\begin{equation}\label{eq:solenoid-return}
 \rho(q)=\inf_{\lambda\in\Z[1/2]}
       \max\{|3q/5-\lambda|,|\lambda|_2\}.
\end{equation}
For a non-integral dyadic rational $\lambda$, $|\lambda|_2\geq2$.
An integer nearest to $3q/5$ gives a value at most one, so only
integers need be considered. Integers with $|3q/5-\lambda|>1$ can
also be discarded. The resulting finite calculation gives
\[
 \begin{array}{c|rrrrrrr}
 q&1&2&3&4&5&6&7\\ \hline
 \rho(q)&3/5&4/5&1/2&1/2&1&2/5&1/4\\
 \text{a minimizing }\lambda&0&2&2&2&3&4&4
 \end{array}
\]
and consequently
\[
 q_1=1,\quad q_2=3,\quad q_3=6,\quad q_4=7
 <q_2+q_3=9.
\]
This is the adelic torus corresponding to $\mathcal P=\{2\}$ in
\cite{DasHaynes2023}. The diagonal subgroup is closed by the separation
argument in Section~\ref{subsec:specializations}. The point $\alpha$ has infinite order: if $(3k/5,0)=(\lambda,\lambda)$,
the second coordinate forces $\lambda=0$ and then $k=0$.

\paragraph{Two real dimensions.}
Let $U=\Z/256\Z$, with
\[
 |\bar m|_U=
 \begin{cases}
  0,&\bar m=0,\\
  2^{-\nu_2(m)},&\bar m\ne0.
 \end{cases}
\]
Here $\nu_2(m)\in\{0,\ldots,7\}$ is independent of the chosen
representative of a nonzero residue class. This is a
translation-invariant ultrametric. Take
\begin{equation}\label{eq:planar-example}
 X=(\R^2\times U)/(\Z^2\times\{0\}),\qquad
 \alpha=\bigl((9/20,8/13),\bar1\bigr)+(\Z^2\times\{0\}).
\end{equation}
For $1\leq q<256$,
\begin{equation}\label{eq:planar-return}
 \rho(q)=\max\{D(q),2^{-\nu_2(q)}\},\qquad
 D(q)^2=\frac{a(q)^2}{400}+\frac{b(q)^2}{169},
\end{equation}
where $a(q)$ and $b(q)$ are the least absolute residues of $9q$
modulo $20$ and of $8q$ modulo $13$, respectively.
The candidate records and their errors are
\[
\begin{array}{c|ccccc}
q&44&80&96&104&120\\ \hline
q(9/20,8/13)\bmod\Z^2
 &(-1/5,1/13)&(0,3/13)&(1/5,1/13)&(-1/5,0)&(0,-2/13)\\
2^{-\nu_2(q)}&1/4&1/16&1/32&1/8&1/8\\
\rho(q)&1/4&3/13&\sqrt{194}/65&1/5&2/13
\end{array}
\]
The return distances in the last row are strictly decreasing.
To check that these are consecutive records, first note that a return
of size at most $1/4$ must have $4\mid q$. After $q=44$, an improvement
must have $8\mid q$. All remaining candidates are covered by the
following exact calculations using \eqref{eq:planar-return}:
\begin{center}
\begin{tabular}{@{}ccc@{}}
\toprule
Interval & Candidates & Minimum of $D(q)^2$\\
\midrule
$1\leq q<44$ & $4,8,12,16,20,24,28,32,36,40$ & $269/4225>1/16$\\
$44<q<80$ & $48,56,64,72$ & $794/4225>1/16$\\
$80<q<96$ & $88$ & $776/4225>9/169$\\
$96<q<104$ & none & ---\\
$104<q<120$ & $112$ & $701/4225>1/25$\\
\bottomrule
\end{tabular}
\end{center}
Thus $44,80,96,104,120$ are consecutive strict best-return
denominators. If $q_n=44$, then
\[
 q_{n+4}=120<44+80=q_n+q_{n+1},
\]
which proves the required failure of the four-step recurrence.
This example already has a direct-product quotient and a finite
ultrametric factor.

At the critical indices of the two examples, minimizing representatives
can be chosen with
\[
 \begin{array}{c|cc}
 &\text{real component norm}&\text{ultrametric component norm}\\ \hline
 q_2=3\text{ in the solenoid}&1/5&1/2=r_2\\
 q_n=44\text{ in the planar example}&\sqrt{194}/65&1/4=r_n.
 \end{array}
\]
Thus these representatives do not satisfy the strict condition in
Proposition~\ref{prop:archimedean}. More generally, their failed
recurrences illustrate that the preceding record cannot always be
used as an additional separated real vector.

\subsection{Further questions}\label{subsec:questions}

Theorem~\ref{thm:abelianization} shows that questions involving only
cyclic-orbit metrics and the effective real dimension are equivalent in
the homogeneous and abelian mixed-quotient settings. In particular, the
five-step planar recurrence no longer requires a normalizer hypothesis.
The unresolved questions below concern mixed geometry itself.

Does every mixed homogeneous space with effective Euclidean dimension two
satisfy
\[
    g_N\leq5?
\]
Corollary~\ref{cor:euclidean} gives $6$. It is equivalent to ask this for
abelian mixed quotients. The failure of the four-step recurrence in
Section~\ref{subsec:examples} rules out the most direct denominator
argument but does not rule out a five-distance theorem.

For effective Euclidean dimension three, does
\[
    q_{n+9}>q_n+q_{n+1}
\]
always hold whenever the left-hand side is defined? Again, it suffices
to consider abelian mixed quotients. Proposition~\ref{prop:weak} gives
the index $13$, whereas the proposed index $9$ would continue the
low-dimensional pattern of Theorem~\ref{thm:lowdim}. The purely real
case satisfies the stronger eight-step non-strict recurrence by
Proposition~\ref{prop:shulga-real}. By Theorem~\ref{thm:growth}, the mixed
nine-step recurrence would imply $g_N\leq10$; the stronger distance bound
$g_N\leq9$ is a separate question.

More generally, Proposition~\ref{prop:weak} gives the recurrence with
index $P(E_\Lambda)+1$. For each effective dimension $m\geq1$ with
$\sigma_m^{>}\leq2^m$, it implies the index $2^m+1$; for $m=0$ the
same conclusion follows from $P(\{0\})=1$. The main unresolved issues
are therefore sharper low-dimensional bounds and the effect of the
ultrametric factor on denominator growth. Abelianization removes
noncommutativity and non-normality as additional obstructions for a
single cyclic orbit.

\paragraph{Acknowledgments.} The author was supported by EPSRC grant UKRI 2768. 

The author used ChatGPT (OpenAI) to assist with language editing. AI, ChatGPT (OpenAI), assistance was used in constructing the proof in Section 2.7 and elaborating the example in the section “A genuinely homogeneous example: a Weyl-twisted $p$-adic suspension”. The author takes full responsibility for the content of the article. The author has independently reviewed and verified the results.

\end{document}